\documentclass[pdftex,pra,aps,onecolumn,twoside,nofootinbib,showkeys]{revtex4}
\pdfoutput=1
\usepackage[utf8]{inputenc}
\usepackage[dvips]{graphicx}
\usepackage[stable]{footmisc}
\usepackage{amsmath,amsthm,amssymb}
\usepackage{youngtab}
\usepackage[all]{xy}
\usepackage{dsfont}
\usepackage{appendix}
\usepackage{enumerate}
\usepackage{youngtab}
\numberwithin{equation}{section}

\newcommand{\<}{\langle}
\renewcommand{\>}{\rangle}
\newcommand\be{\begin{equation}}
\newcommand\ee{\end{equation}}
\newcommand\ot{\otimes}

\newcommand\sgn{\operatorname{sgn}}

\newcommand\rmi{\mathrm{i}\mspace{1mu}}

\newcommand\bea{\begin{array}}
\newcommand\eea{\end{array}}
\newcommand\ben{\begin{eqnarray}}
\newcommand\een{\end{eqnarray}}

\newcommand\bei{\begin{itemize}}
\newcommand\eei{\end{itemize}}
\newcommand\bee{\begin{enumerate}}
\newcommand\eee{\end{enumerate}}

\usepackage{marginnote}

\newcommand{\tr}{{\rm tr}}
\def\ot{\otimes}

\def\bei{\begin{itemize}}
\def\eei{\end{itemize}}

\newtheorem{theorem}{Theorem}[section]
\newtheorem{fact}[theorem]{Fact}
\newtheorem{lemma}[theorem]{Lemma}
\newtheorem{corollary}[theorem]{Corollary}
\newtheorem{example}[theorem]{Example}
\newtheorem{proposition}[theorem]{Proposition}
\newtheorem{definition}[theorem]{Definition}

\newtheorem{remark}[theorem]{Remark}

\newtheorem*{rep@theorem}{\rep@title}
\newcommand{\newreptheorem}[2]{%
\newenvironment{rep#1}[1]{%
 \def\rep@title{#2 \ref{##1} (restatement)}%
 \begin{rep@theorem}}%
 {\end{rep@theorem}}}
\makeatother

\newreptheorem{theorem}{Theorem}
\newreptheorem{lemma}{Lemma}
\newreptheorem{definition}{Definition}

{\par\addvspace{\medskipamount}\noindent\textbf{Examples.}\hspace{1ex}}%
{\par\medskip}

\DeclareUnicodeCharacter{981}{\phi}
\DeclareUnicodeCharacter{348}{{\hat{S}}}
\DeclareUnicodeCharacter{8224}{{\dagger}}

\usepackage{color}

\newcommand\Hom{\operatorname{Hom}}

\begin{document}

\title{Explicit constructions of unitary transformations between equivalent irreducible representations}

\author{Marek Mozrzymas$^1$, Micha{\l} Studzi{\'n}ski$^{2,3}$,
  Micha{\l} Horodecki$^{2,3}$}

\affiliation{
$^1$ Institute for Theoretical Physics, University of Wroc{\l}aw, 50-204 Wroc{\l}aw, Poland\\
$^2$ Faculty of Mathematics, Physics and Informatics, University of Gda{\'n}sk, 80-952 Gda{\'n}sk, Poland\\
$^3$ Quantum Information Centre of Gda\'{n}sk, 81-824 Sopot, Poland}

\date{\today}


\begin{abstract}
Irreducible representations (irreps) of a finite group $G$ are
equivalent if there exists a similarity transformation between them.
In this paper, we describe an explicit algorithm for constructing this
transformation between a pair of equivalent irreps, assuming we are
given an algorithm to compute the matrix elements of these irreps.
Along the way, we derive a generalization of the classical
orthogonality relations for matrix elements of  irreps
of finite groups.
We give an explicit form of such unitary matrices
for the important case of conjugated Young-Yamanouchi representations,
when our group $G$ is symmetric group $S(N)$.   
\end{abstract}

\keywords{irreducible representations, equivalent representations, Young-Yamanouchi basis, symmetric group}

\maketitle 

\section{Introduction}
Group representation theory is a powerful tool in physics for studying
systems with symmetries.   When performing numerical optimization or
simulation of physical systems, representation theory can dramatically
simplify the required calculations.  There are examples of this for
the many-electron problem in physics and quantum
chemistry~\cite{Pauncz,Pauncz2,Pauncz3}, in quantum information
theory~\cite{Eggeling1,springerlink:10.1007/PL00001027,PhysRevLett.82.4344,PhysRevA.84.022320,Whitfield1,Czechlewski1,Cwiklinski1,Studzinski2,Cwiklinski2,Matsumoto1}
and elsewhere.

Motivated by this wide range of possible applications we focus in this
paper on relations between irreducible but equivalent representations
of some finite group $G$. The irreps of finite groups can be calculated  by using for example program GAP~\cite{GAP}. Two equivalent representations
of group $G$ can be mapped to each other by a similarity transformation
determined by some nonsingular matrix $X$. 
So in fact two sets of matrices representing the group elements in each irrep are conjugated. A general method of solving the conjugacy problem for two arbitrary sets   of elements of a finite-dimensional algebra over a large class of fields is described in paper~\cite{Chistov}. The method presented in mentioned paper is based on the solutions systems of linear equations, which yields some linear subspaces and  subalgebras of the algebra and it is shown that the generators the linear subspace (which are in fact a subalgebra module), if they exists, are solutions of the conjugacy problem. In addition in case of the matrix algebra over real algebraic field it is possible, by calculating  a square root of some symmetric positive definite matrix, to construct an orthogonal matrix, which is the solution of the matrix conjugacy problem. The algorithms that lead to  the solution of the conjugacy problem are  of the polynomial time.
In this paper we show that in the case when we have to deal with very special sets of matrices representing the group elements in two equivalent irreducible representations it is possible to construct an unitary conjugation of these sets in a different way, using very particular properties of irreducible representations of finite groups, in particular using the orthogonality relations for irreps. As a  result we derive an explicit formula for the unitary matrix defining the similarity transformation. In our method, instead of solving the systems of linear equations one has to find a nonzero normalization factor in the formula for the unitary matrix, which is also a problem of polynomial time, but it seems to be easier to calculate. This is because due to use of the group properties, we have to deal with smaller number of equations. We give several examples how this method works in practice for the permutation group $S(n)$. 
 We also analyze the similarity
transforms for a class of equivalent pairs of permutation-group irreps
and show that the
transformation matrices have a very simple anti-diagonal form. Using
general  results of the construction we also formulate a
generalization of the well-known classical orthogonality relations for
irreps of a finite group $G$. 

This paper is organized as follows. In the Section~\ref{preliminaries} we formulate the problem and recall basic statements from group representation theory, which play important role in next sections.
In the Section~\ref{gen} we describe an explicit method
 to compute unitary transformation
matrices between two irreducible but equivalent representations for
some finite group $G$, and we present the full solution of the problem with details and discussion. In particular we show a few
interesting facts regarding some properties of such unitary
transformations (doubly stochastic property, generalized orthogonality
property for irreducible representations etc.). In
Section~\ref{symm} we apply results from Section~\ref{gen} to the
symmetric group $S(N)$ and present examples 
for $N=3,4,5,6$ which show how our algorithm works in practice.
Next, in the Section~\ref{symm2} we state and prove Theorem~\ref{hypo1} and Proposition~\ref{MH} in
which say that the unitary matrix which maps
conjugated Young-Yamanouchi irreps of $S(N)$ consist simply of 
$\pm 1$ entries along the anti-diagonal. Finally in the Section~\ref{appl} we present some mathematical application of our result to projector onto specific subspace suggested by Schur-Weyl duality~\cite{RWallach}.

\section{Preliminaries}
\label{preliminaries}
In this section we give some basic ideas regarding similarity transformation between irreducible and equivalent irreps of some finite group $G$. Most of the informations in this chapter is taken from~\cite{Procesi,Fulton,RWallach,JQChen}. We start from the following definition:
\begin{definition}
\label{th:general}
We say that two different  irreducible representation (irreps) $\vartheta $ and $\psi $  of the finite group $G$ are equivalent when
\be
\label{1}
\exists X\in \operatorname{GL}(n,\mathbb{C}):X^{-1}\vartheta (g)X=\psi (g),\quad \forall g\in G, 
\ee
where $n$ is the dimension of the irreps $\vartheta $ and $\psi $.
\end{definition}
Our task is to find an explicit formula for the transformation matrices $X$ from Definition~\ref{th:general}.

The form of matrix $X\in \operatorname{GL}(n,\mathbb{C})$ in equation~\eqref{1} in Definition~\ref{th:general} is strongly restricted by the group $G$ and its representations $\vartheta$ and $\psi$. In fact we have 

\begin{proposition}
\label{prop4}
Suppose that the matrix representations $\vartheta, \psi$ (which are not necessarily unitary) of finite group $G$ are irreducible and equivalent, then the matrix $X\in \operatorname{GL}(n,\mathbb{C})$ which satisfies
\be
\forall  g\in G \quad X^{-1}\vartheta(g) X=\psi(g)
\ee
is unique up to non-zero scalar multiple.
\end{proposition}

This statement is a corollary of the following 

\begin{theorem}
\label{thm:MM}
Let $\vartheta, \psi$ be equivalent matrix irreps of $G$ (not necessarily unitary). Then the map
\be
\label{eqMM1}
\begin{split}
&\forall  g \in G \quad \Psi:g \rightarrow \operatorname{Aut}\left(\operatorname{M}(n,\mathbb{C})\right) \\
&\forall  X \in \operatorname{M}(n,\mathbb{C}) \quad \Psi(g)(X)=\vartheta(g)X\psi(g^{-1})
\end{split}
\ee
defines a representation of the group $G$ in the linear space $\operatorname{M}(n,\mathbb{C})$ over $\mathbb{C}$. The representation $\Psi$ is reducible and the one-dimensional identity representation of $G$ is included in $\Psi$ only once i.e. there is only one, up to scalar multiple, matrix $X$ which satisfies
\be
\forall g\in G \quad \vartheta(g)X\psi(g^{-1})=X.
\ee
\end{theorem}


\begin{remark}
In the particular case when $\vartheta=\psi$ the matrix $X$ is up to a scalar multiple equal to \noindent
\(\mathds{1}\), which in this case generates the identity irrep in the representation $\Psi$, then the statement of the Theorem~\ref{thm:MM} follows directly from the Schur's Lemma.
\end{remark}

Now let us come back to the Proposition~\ref{prop4}. The equation~\eqref{1} from Definition~\ref{th:general} may be written in the form
\be
\forall g\in G \quad \vartheta(g)X\psi(g^{-1})=X
\ee
given in the Theorem~\ref{thm:MM} and this theorem states that such a matrix $X$ is unique up to scalar multiple.

If the irreps $\vartheta, \psi$ are unitary then the matrix $X$ may be chosen to be unitary, in fact we
have

\begin{lemma}
\label{problem:1}
If $\vartheta $ and $\psi $ are two different but equivalent unitary and
irreducible matrix representations, in $\operatorname{M}(n,\mathbb{C}),$ of a finite group $G$, then
\be
\label{eq:0}
\exists U\in U(n):U^{\dagger}\vartheta (g)U=\psi (g),\quad \forall g\in G. 
\ee
\end{lemma}

\begin{proof}
\bigskip We have 
\be
\label{2}
X^{-1}\vartheta (g)X=\psi (g),~\forall g\in G\quad \Leftrightarrow \quad
\vartheta (g)X=X\psi (g),~\forall g\in G
\ee
and from the unitarity of $\vartheta $ and $\psi $ we get 
\be
\label{3}
X^{\dagger}\vartheta (g^{-1})=\psi (g^{-1})X^{\dagger},\quad \forall g\in G
\ee
and 
\be
\label{4}
X^{\dagger}X=\psi (g^{-1})X^{\dagger}X\psi (g),~\forall g\in G\quad \Leftrightarrow
\quad \psi (g)X^{\dagger}X=X^{\dagger}X\psi (g),~\forall g\in G.
\ee
The irreducibility of the representation $\psi $ and the Schur Lemma implies 
\be
\label{5}
X^{\dagger}X=\alpha \mathbf{1}_{n}:\alpha >0\Rightarrow \quad X^{\dagger}=\alpha X^{-1}.
\ee
Define 
\be
\label{6}
U=\frac{1}{\sqrt{\alpha }}X\Rightarrow U^{\dagger}U=\frac{1}{\alpha }X^{\dagger}X=%
\mathbf{1}_{n},
\ee
so the matrix $U$ is unitary and satisfies%
\be
U^{\dagger}\vartheta (g)U=\sqrt{\alpha }X^{-1}\vartheta (g)\frac{1}{\sqrt{\alpha }}%
X=\psi (g),\quad \forall g\in G.
\ee
\end{proof}

In the following we will assume that irreps $\vartheta$ and  $\psi$ are unitary and our task is to find an explicit formula for the unitary matrix $U,$
which gives the similarity transformation between the representations $%
\vartheta $ and $\psi $. Such a matrix $U$ is not unique and we have

\begin{lemma}
\label{lemma:2}
If $U$ is such that 
\be
\label{eq:2}
U^{\dagger}\vartheta (g)U=\psi (g),\quad \forall g\in G,
\ee
then $U^{\prime }=\operatorname{e}^{\rmi\mu }U:\mu \in \mathbb{R}$ also defines an unitary similarity between $\vartheta $ and $\psi .$ 
\be
\label{eq:3}
U^{\prime \dagger}\vartheta (g)U^{\prime }=\psi (g),\quad \forall g\in G
\ee
\end{lemma}
In the next section we show how to construct unitary transformation matrices $U$ which map between two equivalent irreps of some finite group $G$.

\section{General method of construction}
\label{gen}
In this section we present explicit construction method of the unitary matrices which represent similarity transformation between irreducible but equivalent representations of some finite group $G$.
 
In order to derive the formula for the matrix $U$, we consider the equation~\eqref{eq:2} which contains all condition on $U$.
The $\operatorname{RHS}$ of the equation~\eqref{eq:2} for $U=(u_{ij})$ may written in the following way
\be
\sum_{st}u_{bs}^{\dagger}\vartheta _{st}(g)u_{tj}=\sum_{st}\overline{u}
_{sb}u_{tj}\vartheta _{st}(g)=\sum_{st}(\overline{U}\otimes U)_{st,bj}\vartheta
_{st}(g), 
\ee
therefore the equation for $U=(u_{ij})$ takes the form 
\be
\sum_{st}(\overline{U}\otimes U)_{st,bj}\vartheta _{st}(g)=\sum_{st}(
\overline{U}\otimes U)_{bj,st}^{t}\vartheta _{st}(g)=\psi_{bj}(g),\quad
\forall g\in G, 
\ee
where we have used 
\be
X=(x_{ij}),\quad Y=(y_{kl})\Rightarrow (X\otimes Y)_{jk,il}=x_{ji}y_{kl}. 
\ee
In the next step we use the orthogonality relations for the irreducible
representations of finite groups which may be formulated as follows

\begin{proposition}
\label{prop44}
\bigskip\ Suppose that $\vartheta$ and $\psi $ are two  
irreducible matrix representations of a finite group $G.$ Then
\be
\label{ort}
\sum_{g\in G}\psi _{ij}(g)\vartheta _{kl}(g^{-1})=
\begin{cases}
0 & \text{if $\psi$ and $\vartheta$ are inequivalent}\\
\frac{|G|}{n}\delta _{jk}\delta _{il}
& \text{if $\psi$ and $\vartheta$ are equal,}
\end{cases}
\ee
where by $|G|$ we denote cardinality of group $G$. Equation~\eqref{ort} does not apply if  $\psi$ and $\vartheta$ are equivalent
but not equal.
\end{proposition}

Using this Proposition we get 
\be
\sum_{jk}\sum_{g\in G}(\overline{U}\otimes U)_{bj,{}st}^{t}\vartheta
_{st}(g)\vartheta _{ia}(g^{-1})=\sum_{g\in G}\psi _{bj}(g)\vartheta
_{ia}(g^{-1}),
\ee
\be
(\overline{U}\otimes U)_{bj,{}ai}^{t}=\frac{n}{|G|}\sum_{g\in
G}\psi _{bj}(g)\vartheta _{ia}(g^{-1})=\frac{n}{|G|}\sum_{g\in G}\psi _{bj}(g)%
\overline{\vartheta _{ai}(g)},
\ee
and finally we get the equation where the desired matrix $U$ and the
given representations $\vartheta $ and $\psi $ are separated:
\be
(\overline{U}\otimes U)_{ai,{}bj}=\frac{n}{|G|}\sum_{g\in G}\overline{\vartheta
_{ai}(g)}\psi _{bj}(g)\equiv A_{ai,{}bj},\qquad \overline{U}\otimes U=A.
\ee
Now we have to extract the matrix $U$ from this equation.

This equation is, in fact, the matrix equation in $\operatorname{M}(n^2,\mathbb{C})$ and on $\operatorname{LHS}$ we have tensor product block structure where the blocs are
of the form
\be
\label{eq18}
\overline{u_{ab}}U=A_{ab}:U=(u_{ij}),{}A_{ab}=(A_{ab})_{ij}\in \operatorname{M}(n,\mathbb{C})
\ee
and if $\overline{u}_{ab}=r_{ab}\operatorname{e}^{\rmi\mu _{ab}}\neq 0$ then 
\be
\operatorname{e}^{\rmi\mu _{ab}}u_{ij}=\frac{1}{r_{ab}}\frac{n}{|G|}\sum_{g\in G}\overline{
\vartheta _{ai}(g)}\psi _{bj}(g), 
\ee
where, from Lemma~\ref{lemma:2} the matrix $U^{\prime }=\operatorname{e}^{\rmi\mu _{ab}}U$ also gives the
similarity transformation that we are looking for. Thus in order to get the
explicit formula for a unitary matrix connecting $\vartheta $ and $\psi $
by the similarity transformation we have to know for which indices $(a,b)$
the weight $r_{ab}$ is not equal to zero. From the equation~\eqref{eq18} we get 
\be
r_{ab}\sqrt{n}=||A_{ab}||_{\tr},\qquad ||A_{ab}||_{\tr}=\sqrt{%
\tr(A_{ab}^{\dagger}A_{ab})}, \qquad \text{and} \qquad ||U||_{\tr}=\sqrt{n}.
\ee
which obviously shows that if $\overline{u}_{ab}=0$ then the corresponding
block $A_{ab}$ in $A$ is a zero matrix. On the other hand direct calculation gives 
\be
||A_{ab}||_{\tr}=\frac{n}{\sqrt{|G|}}\left( \sum_{g\in G}\vartheta _{aa}(g)\psi
_{bb}(g^{-1})\right) ^{\frac{1}{2}} 
\ee
therefore 
\be
r_{ab}=\sqrt{\frac{n}{|G|}}\left( \sum_{g\in G}\vartheta _{aa}(g)\psi
_{bb}(g^{-1})\right) ^{\frac{1}{2}}. 
\ee
The weight $r_{ab}$, as a function of indices $(a,b),$ indicates which
elements $u_{ab}$ of the matrix $U$ are non-zero and consequently
which blocks $A_{ab}\in \operatorname{M}(n,\mathbb{C})$ in the block
matrix $A$ are non-zero.

Summarizing we get

\begin{theorem}
\label{main:thm}
Suppose that $\vartheta $ and $\psi $ are two different but equivalent unitary
and irreducible matrix representations, in $\operatorname{M}(n,\mathbb{C}),$ of a finite group $G.$ Then
\begin{enumerate}[1)]
\item there exist indices $a,b=1,\ldots,n=\dim \psi$ such that 
\be
\sum_{g\in G}\vartheta _{aa}(g)\psi _{bb}(g^{-1})>0,
\ee

\item the matrix $U=(u_{ij})$ that determines the similarity transformation 
\be
\label{item2}
U^{\dagger}\vartheta (g)U=\psi (g),\quad \forall g\in G, 
\ee
has the following form
\be
\label{thm11}
u_{ij}\equiv u_{ij}(ab)=\frac{1}{r_{ab}}\frac{n}{|G|}\sum_{g\in G}\overline{
\vartheta _{ai}(g)}\psi _{bj}(g),
\ee
where 
\be
\label{thm12}
r_{ab}=\sqrt{\frac{n}{|G|}}\left( \sum_{g\in G}\vartheta _{aa}(g)\psi
_{bb}(g^{-1})\right) ^{\frac{1}{2}} 
\ee
and the $(a,b)$ are chosen in such a way that $r_{ab}>0$, which is possible
from the statement 1).
\end{enumerate}
\end{theorem}

\begin{remark}
In the point 1) of Theorem~\ref{main:thm} there are maximally $n$ equations that we need to check to find a non-zero factor $r_{ab}$. The problem of finding non-zero coefficient $r_{ab}$ can be realized in the polynomial time.
\end{remark}

\begin{remark}
If the representations $\vartheta$ and $\psi$ are orthogonal then the matrix $U$ is also orthogonal.
\end{remark}

\begin{remark}
For a fixed values $a,b=1,\ldots,n$ the unitary matrix $U=U(ab)$ in equation~\eqref{thm11} in Theorem~\ref{main:thm} is determined in a unique way by irreps $\psi, \vartheta$. From Lemma~\ref{lemma:2} it follows that for arbitrary $\alpha \in \mathbb{R}$ the unitary matrix $U'=\operatorname{e}^{\rmi \alpha}U(ab)$ also gives the similarity transformation equation~\ref{item2}.
\end{remark}

\begin{remark}
From the Eq.~\eqref{thm11} and the orthogonality relations (Proposition~\ref{prop44}) it follows that if the
irreducible representations $\vartheta $ and $\psi $ were not equivalent then $%
U=0,$ in agreement with Schur Lemma.
\end{remark}

\begin{remark}
\label{stoch}
From the unitarity of the matrix $U=(u_{ab})$ it follows that
\be
\forall a=1,\ldots,n \quad \sum_b r_{ab}^2=1, \quad \forall b=1,\ldots,n \quad \sum_{a} r_{ab}^2=1,
\ee
so the matrix with elements $r_{ab}^2$ is double stochastic.
\end{remark}

From Theorem~\ref{main:thm} in particular from equation~\eqref{thm11} one can deduce the following corollary which is a generalization of the classical orthogonality relation for irreps of finite group $G$ given in Proposition~\ref{prop44}, which plays very important role in the theory of group representation.

\begin{corollary}
\label{ortRel}
Let $\psi$, $\vartheta$ be unitary, different but equivalent irreps of $G$. Then there exists an unitary matrix $U=(u_{ab}) \ : \ r_{ab}=|u_{ab}|$ such that
\be
\label{llll}
r_{ab}u_{ij}=\frac{n}{|G|}\sum_{g \in G} \vartheta_{ia}(g^{-1})\psi_{bj}(g).
\ee
In particular, when $\vartheta=\psi$ then $U=\text{\noindent\(\mathds{1}\)}$ and thus formula~\eqref{llll} takes the form of classical orthogonality relation for irrep $\vartheta$
\be
\frac{n}{|G|}\sum_{g\in G} \vartheta_{ia}(g^{-1})\vartheta_{bj}(g)=\delta_{ab}\delta_{ij}.
\ee
\end{corollary}

\section{Examples regarding symmetric group}
\label{symm}
 In this section we show a few examples of unitary matrices by application of the Theorem~\ref{main:thm} (up to global phase) for the symmetric group $S(N)$ for some small $N$. We start from the simplest examples for $S(3)$.
\begin{example}
\label{exx1}
Consider two different but equivalent representations of the group $S(3)$ 
\be
\psi ^{\varepsilon }(12)=\left( 
\begin{array}{cc}
0 & 1 \\ 
1 & 0%
\end{array}%
\right) ,\quad \psi ^{\varepsilon }(13)=\left( 
\begin{array}{cc}
0 & \varepsilon  \\ 
\varepsilon ^{-1} & 0%
\end{array}%
\right) ,\quad \psi ^{\varepsilon }(23)=\left( 
\begin{array}{cc}
0 & \varepsilon ^{-1} \\ 
\varepsilon  & 0%
\end{array}%
\right) ,\quad 
\ee
\be
\psi ^{\varepsilon }(123)=\left( 
\begin{array}{cc}
\varepsilon  & 0 \\ 
0 & \varepsilon ^{-1}%
\end{array}%
\right) ,\quad \psi ^{\varepsilon }(132)=\left( 
\begin{array}{cc}
\varepsilon ^{-1} & 0 \\ 
0 & \varepsilon 
\end{array}%
\right) ,\quad 
\ee
where $\varepsilon ^{3}=1,$ and 
\be
\varphi (12)=\left( 
\begin{array}{cc}
1 & 0 \\ 
0 & -1%
\end{array}%
\right) ,\quad \varphi (13)=\left( 
\begin{array}{cc}
-\frac{1}{2} & -\frac{\sqrt{3}}{2} \\ 
-\frac{\sqrt{3}}{2} & \frac{1}{2}%
\end{array}%
\right) ,\quad \varphi (23)=\left( 
\begin{array}{cc}
-\frac{1}{2} & \frac{\sqrt{3}}{2} \\ 
\frac{\sqrt{3}}{2} & \frac{1}{2}%
\end{array}%
\right) ,\quad 
\ee
\be
\varphi (123)=\left( 
\begin{array}{cc}
-\frac{1}{2} & \frac{\sqrt{3}}{2} \\ 
-\frac{\sqrt{3}}{2} & -\frac{1}{2}%
\end{array}%
\right) ,\quad \varphi (132)=\left( 
\begin{array}{cc}
-\frac{1}{2} & -\frac{\sqrt{3}}{2} \\ 
\frac{\sqrt{3}}{2} & -\frac{1}{2}%
\end{array}%
\right) .\quad 
\ee
Applying the theorem we get 
\be
r_{11}=\sqrt{\frac{n}{|G|}}\left( \sum_{g\in S(3)}\varphi _{11}(g)\psi
_{11}(g^{-1})\right) ^{\frac{1}{2}}=\frac{1}{\sqrt{2}},
\ee
and
\be
U=\frac{1}{\sqrt{2}}\left( 
\begin{array}{cc}
1 & 1 \\ 
\frac{1}{\sqrt{3}}(\varepsilon -\overline{\varepsilon }) & \frac{-1}{\sqrt{3}%
}(\varepsilon -\overline{\varepsilon })%
\end{array}%
\right) .
\ee
\end{example}

\begin{example}
\label{exx2}
It is clear that the representations $\psi ^{\varepsilon }$ and $\varphi
=\psi ^{\overline{\varepsilon }}$ are equivalent. In this case the theorem
gives 
\be
r_{11}=0,\qquad r_{12}=1, 
\ee
and 
\be
U=\left( 
\begin{array}{cc}
0 & 1 \\ 
1 & 0%
\end{array}%
\right) , 
\ee
which is obvious without applying the theorem.
\end{example}

\begin{example}
\label{exx3}
It is also known that for $S(3)$ the representations $\psi ^{\varepsilon }$
and $\varphi =\mathbf{sgn}\psi ^{\varepsilon }$ are equivalent~\footnote{Here we use the following convention. Namely by $\mathbf{sgn}$ we understand so called signum representation for the group $S(n)$. By $\sgn(\sigma)$ for some $\sigma \in S(n)$ we understand parity of the permutation $\sgn : S(n) \rightarrow \{-1,1\}$, which is defined as follows 
\[
\forall \sigma \in S(n) \ \sgn(\sigma)=\left( -1\right)^{N(\sigma)},
\]
where $N(\sigma)$ is the number of inversion in $\sigma$.}. again applying the
theorem we get%
\be
r_{11}=1 
\ee
and 
\be
U=\left( 
\begin{array}{cc}
1 & 0 \\ 
0 & -1%
\end{array}%
\right) . 
\ee
\end{example}

Now we present a few examples of unitary matrices from Lemma~\ref{problem:1} for irreducible representations of symmetric groups $S(N)$ for some small $N$ using directly formulas~\eqref{thm11} and~\eqref{thm12} from Theorem~\ref{main:thm}. These matrices map conjugated irreps calculated in Young-Yamanouchi basis which we describe further in this section. To obtain  this results we wrote code in \textit{Mathematica 7} and examples~\ref{exx4},~\ref{exx5} and~\ref{exx6} are calculated using Young-Yamanouchi formalism (here we refer reader to the~\cite{Procesi,Fulton,RWallach,JQChen} or further part of this paper).
\begin{example}
\label{exx4}
In this example we present unitary transformations between conjugated Young-Yamanouchi irreps for the symmetric group $S(4)$. We restrict our attention to partitions $\lambda_1=(3,1)$ and $\lambda_2=(2,2)$, so it means that our unitary matrices transform irreps on partitions $\lambda_i$ to irreps on $\sgn \lambda_i^{t}$, where $i=1,2$. We will use this convention also in next example.
\be
U_{\lambda_1}=\begin{pmatrix} {} & {} & -1\\ {} & 1 & {}\\ -1 & {} & {} \end{pmatrix}, \quad  U_{\lambda_2}=\begin{pmatrix}{} & -1\\ 1 & {} \end{pmatrix}
\ee
\end{example}

\begin{example}
\label{exx5}
In this example we present unitary transformations between conjugated Young-Yamanouchi irreps for the symmetric group $S(5)$. We restrict our attention to partitions $\lambda_1=(4,1)$, $\lambda_2=(3,2)$ and $\lambda_3=(3,1,1)$.
\be
U_{\lambda_1}=\begin{pmatrix} {} & {} & {} & -1\\ {} & {} & 1 & {}\\ {} & -1 & {} & {}\\ 1 & {} & {} & {} \end{pmatrix}, \quad U_{\lambda_2}=\begin{pmatrix}{} & {} & {} & {} & -1\\ {} & {} & {} & 1 & {}\\ {} & {} & 1 & {} & {}\\ {} & -1 & {} & {} & {}\\ 1 & {} & {} & {} & {} \end{pmatrix}, \quad U_{\lambda_3}=\begin{pmatrix} {} & {} & {} & {} & {} & -1\\ {} & {} & {} & {} & 1 & {}\\ {} & {} & {} & -1 & {} & {}\\
{} & {} & -1 & {} & {} & {}\\ {} & 1 & {} & {} & {} & {}\\ -1 & {} & {} & {} & {} & {}\end{pmatrix}.
\ee
\end{example}

\begin{example}
\label{exx6}
In this example we present unitary transformations between conjugated Young-Yamanouchi irreps for the symmetric group $S(6)$. We restrict our attention to partitions $\lambda_1=(5,1)$, $\lambda_2=(4,2)$, $\lambda_3=(4,2)$, $\lambda_4=(4,1,1)$, $\lambda_5=(3,3)$ and $\lambda_6=(3,2,1)$.
\be
U_{\lambda_1}=\begin{pmatrix} {} & {} & {} & {} & -1 \\ {} & {} & {} & 1 & {} \\ {} & {} & -1 & {} & {} \\ {} & 1 & {} & {} & {} \\ -1 & {} & {} & {} & {} \end{pmatrix}, \quad U_{\lambda_2}=\begin{pmatrix} {} & {} & {} & {} & {} & {} & {} & {} & -1 \\ {} & {} & {} & {} & {} & {} & {} & 1 & {} \\
{} & {} & {} & {} & {} & {} & 1 & {} & {}\\ {} & {} & {} & {} & {} & -1 & {} & {} & {}\\ {} & {} & {} & {} & 1 & {} & {} & {} & {} \\ {} & {} & {} & -1 & {} & {} & {} & {} & {} \\ {} & {} & 1 & {} & {} & {} & {} & {} & {}\\ {} & -1 & {} & {} & {} & {} & {} & {} & {}\\ 1 & {} & {} & {} & {} & {} & {} & {} & {}\end{pmatrix}
\ee
\be
U_{\lambda_3}=\begin{pmatrix} {} & {} & {} & {} & {} & {} & {} & {} & {} & -1 \\ {} & {} & {} & {} & {} & {} & {} & {} & 1 & {} \\ {} & {} & {} & {} & {} & {} & {} & -1 & {} & {} \\ {} & {} & {} & {} & {} & {} & -1 & {} & {} & {} \\ {} & {} & {} & {} & {} & 1 & {} & {} & {} & {} \\ {} & {} & {} & {} & -1& {} & {} & {} & {} & {} \\ {} & {} & {} & 1 & {} & {} & {} & {} & {} & {} \\ {} & {} & -1 & {} & {} & {} & {} & {} & {} & {} \\ {} & 1 & {} & {} & {} & {} & {} & {} & {} & {} \\ -1 & {} & {} & {} & {} & {} & {} & {} & {} & {} \\ \end{pmatrix}, \quad U_{\lambda_5}=\begin{pmatrix} {} & {} & {} & {} & -1  \\ {} & {} & {} & 1 & {} \\ {} & {} & 1 & {} & {} \\ {} & -1 & {} & {} & {} \\ 1 & {} & {} & {} & {} \\\end{pmatrix}
\ee
and finally
\be
U_{\lambda_6}=\begin{pmatrix}{} & {} & {} & {} & {} & {} & {} & {} & {} & {} & {} & {} & {} & {} & {} & -1\\ {} & {} & {} & {} & {} & {} & {} & {} & {} & {} & {} & {} & {} & {} & 1 & {}\\ {} & {} & {} & {} & {} & {} & {} & {} & {} & {} & {} & {} & {} & -1 & {} & {}\\ {} & {} & {} & {} & {} & {} & {} & {} & {} & {} & {} & {} & -1 & {} & {} & {}\\ {} & {} & {} & {} & {} & {} & {} & {} & {} & {} & {} & 1 & {} & {} & {} & {}\\ {} & {} & {} & {} & {} & {} & {} & {} & {} & {} & 1 & {} & {} & {} & {} & {}\\ {} & {} & {} & {} & {} & {} & {} & {} & {} & -1 & {} & {} & {} & {} & {} & {}\\ {} & {} & {} & {} & {} & {} & {} & {} & 1 & {} & {} & {} & {} & {} & {} & {}\\ {} & {} & {} & {} & {} & {} & {} & 1 & {} & {} & {} & {} & {} & {} & {} & {}\\ {} & {} & {} & {} & {} & {} & -1 & {} & {} & {} & {} & {} & {} & {} & {} & {}\\ {} & {} & {} & {} & {} & 1 & {} & {} & {} & {} & {} & {} & {} & {} & {} & {}\\ {} & {} & {} & {} & 1 & {} & {} & {} & {} & {} & {} & {} & {} & {} & {} & {}\\ {} & {} & {} & -1 & {} & {} & {} & {} & {} & {} & {} & {} & {} & {} & {} & {}\\ {} & {} & -1 & {} & {} & {} & {} & {} & {} & {} & {} & {} & {} & {} & {} & {}\\ {} & 1 & {} & {} & {} & {} & {} & {} & {} & {} & {} & {} & {} & {} & {} & {}\\ -1 & {} & {} & {} & {} & {} & {} & {} & {} & {} & {} & {} & {} & {} & {} & {}\end{pmatrix}
\ee
\end{example}
These examples calculated in Young-Yamanouchi basis suggest that all unitary matrices representing similarity relation between Young-Yamanouchi conjugated irreps have quite simply, anti-diagonal form with $\pm 1$ only. In the next section we proof this conjecture. As we mentioned in introduction our method from the Section~\ref{gen} can be applied to any finite group G for which we have characterization of its irreps. To obtain desired matrix elements of irreps one can use for example program GAP~\cite{GAP}, which returns a list of representatives of the irreducible matrix representations of $G$ over field $F$, up to equivalence.

\section{Analytical formula for similarity relation}
\label{symm2}
Main goal of this section is to prove that all unitary matrices which map Young-Yamanouchi conjugated irreps (labelled by $\lambda$ and $\lambda^t$, where $\lambda$ is partition of $N \in \mathbb{N}$, see Equation~\eqref{partitofN} below) have anti-diagonal form with $\pm 1$.
Namely we will show that 
unitary matrix $U$ which transforms irreducible representation $D^{\lambda}$ of $S(N)$ calculated in Young-Yamanouchi basis onto equivalent irreducible representation $\mathbf{sgn} D^{\lambda^t}$ have anti-diagonal form with $\pm 1$ only. In fact the matrix $U$ is of the form
\be
\label{formU}
U=\left( 
\begin{array}{cccccc}
0 & 0 & . & . & . & \sgn(\sigma_{1}) \\ 
0 &  &  &  & \sgn(\sigma_{2}) & 0 \\ 
. &  &  & . &  & . \\ 
. &  & . &  &  & . \\ 
0 & \sgn(\sigma_{d_{\lambda }-1}) &  &  &  & 0 \\ 
\sgn(\sigma_{d_{\lambda }}) & 0 & . & . & 0 & 0%
\end{array}%
\right),
\ee
where the permutations $\sigma_i$ are described in  Proposition~\ref{M11}, Equation~\eqref{M11eq2} below. We show also Proposition~\ref{MH} which states that our unitary transformation can be written as 
\be
U=\sum_{T_{\lambda}} \sgn(T_{\lambda})|T_{\lambda^t}\>\< T_{\lambda}|,
\ee
where $\sgn(T_{\lambda})=\sgn(\sigma)$ and $\sigma$ is the permutation that transforms arbitrary chosen, fixed $\operatorname{SYT}$ $T_{\lambda}^{\times}$ into $T_{\lambda}$ (see Prop.~\ref{M41} and Rem.~\ref{M41A}). We also argue that for different choices of $T_{\lambda}^{\times}$, the corresponding $U$ may differ by a global sign.

In the next part of this section we will prove the above statements. In order to do this firstly we have to introduce briefly the
concept of of irreducible representations of the group $S(N)$ based on the
concepts of natural Young representation and Yamanouchi symbols therefore we
call such a representations Young-Yamanouchi representations.

As it is known any irreducible representation of the group $S(N)$ is
uniquely determined by a partition $\lambda =(\lambda _{1},\lambda
_{2},\ldots,\lambda _{k}),$ where 
\be
\label{partitofN}
\lambda _{1}\geq \lambda _{2}\geq \ldots \geq \lambda _{k}\geq 0,\quad
\sum_{i=1}^{k}\lambda _{i}=n. 
\ee
To each partition $\lambda =(\lambda _{1},\lambda _{2},\ldots,\lambda _{k})$ is
associated a Young diagram $(\operatorname{YD})$ also called  Young frame (see Example~\ref{young1}),
with $\lambda _{i}$ boxes in the $i$-th row, the rows of boxes lined up on
the left. The conjugated partition $\lambda ^{t}=(\lambda _{1}^{t},\lambda
_{2}^{t},\ldots,\lambda _{l}^{t})$ to the partition $\lambda =(\lambda
_{1},\lambda _{2},\ldots,\lambda _{k})$ is defined by interchanging rows and
columns of the Young diagram $\lambda =(\lambda _{1},\lambda
_{2},\ldots,\lambda _{k}).$ For example if $\lambda =(3,2,2,1)$ then $\lambda
^{t}=(4,3,1).$ In general for an arbitrary partition $\lambda=\left(\lambda_1,\lambda_2,\ldots, \lambda_k\right)$ we can obtain conjugate partition $\lambda^t$ using formula
\begin{equation}
\left(\lambda_1, \lambda_2,\ldots, \lambda_k \right)^t=\left(k^{\lambda_k},(k-1)^{\lambda_{k-1}-\lambda_k},\ldots,2^{\lambda_2-\lambda_3},1^{\lambda_1-\lambda_2} \right),
\end{equation} 
where notation $j^m$ denotes that integer $j$ is to be repeated $m$ times with $m=0$ meaning no occurrence.

\begin{example}
\label{young1}
In this example we show explicitly all Young diagrams for $N=4$ together with corresponding partitions $\lambda$.
\begin{center}
\begin{tabular}{ccccccccc}
$\yng(1,1,1,1)$ & &$\yng(2,1,1)$ & &$\yng(2,2)$ & &$\yng(3,1)$ & &$\yng(4)$ \\
&&&&&&&&\\
$\lambda=(1,1,1,1)$, & &$\lambda=(2,1,1)$, & &$\lambda=(2,2)$, & &$\lambda=(3,1)$, & &$\lambda=(4)$
\end{tabular}
\end{center}
\end{example}

\begin{definition}
\label{M1}
A Young tableau $(\operatorname{YT})$ of partition $\lambda =(\lambda
_{1},\lambda _{2},\ldots,\lambda _{k}),$ is a Young diagram $\lambda $ in which
the boxes are fulfilled bijectively by numbers $\{1,2,\ldots,n\}.$ Young
tableau will be denoted $\overline{T}_{\lambda }=(\overline{t}%
_{ij}^{\lambda }),$ where $\overline{t}_{ij}^{\lambda }$ $\in \{1,2,\ldots,n\}$
denote the entry of $\overline{T}_{\lambda }$ in the position $(i,j).$ There are $n!$ of $\operatorname{YT}$.

A standard Young tableau $(\operatorname{SYT})$ is Young tableau where the numbers $%
\{1,2,\ldots,n\}$ appears in the rows of the tableau in the increasing to the
right sequences and in the columns of the tableau in the increasing
sequences from the top to downwards. $\operatorname{SYT}$ will be denoted $T_{\lambda
}=(t_{ij}^{\lambda }).$ The conjugated standard Young tableau $T_{\lambda
^{t}}^{t}$ to the standard Young tableau $T_{\lambda }$ is defined by
interchanging rows and columns (together with the numbers contained in
them) of the standard Young tableau $T_{\lambda }$. Thus the conjugated
standard Young tableau $T_{\lambda ^{t}}^{t}$ is a standard Young tableau
for the conjugated partition $\lambda ^{t}=(\lambda _{1}^{t},\lambda
_{2}^{t},\ldots,\lambda _{l}^{t}).$
\end{definition}


\begin{example}
\label{example1}
Here we present Young tableaux ($\operatorname{YT}$) and standard Young tableaux $(\operatorname{SYT})$ for $N=3$ and partition $\lambda=(2,1)$.
\begin{center}
\begin{tabular}{ccccccccc}
\text{All Young tableaux:\phantom{ppp}}&$\young(12,3) \ $ & $\young(13,2) \ $ & $\young(21,3) \ $ & $\young(23,1) \ $ & $\young(31,2) \ $ &  $\young(32,1)$\\
&&&&&&\\
\text{All standard Young tableaux:\phantom{ppp}}&$\young(12,3) \ $ & $\young(13,2)$ &&&
\end{tabular}
\end{center}

\end{example}

\bigskip

Now one can define, in a natural way, the action of the group $S(N)$ on the
set of all $\operatorname{YT}$.

\begin{definition}
\label{M4}
\be
\forall \sigma \in S(N)\quad \sigma (\overline{T}_{\lambda })=\sigma (%
\overline{t}_{ij}^{\lambda })\equiv (\sigma (\overline{t}_{ij}^{\lambda })) 
\ee
e.i. a permutation $\sigma \in S(N)$ acts on each entry of Young tableau $%
\overline{T}_{\lambda }$ .
\end{definition}

Note that this action of the group $S(N)$ is well defined on the set of all $%
\operatorname{YT}$ and it is not well defined on the subset of $\operatorname{SYT}$ because the action of $%
\sigma \in S(N)$ on standard Young tableau $T_{\lambda }$ may give a $\operatorname{YT}$
which is not a $\operatorname{SYT}$. For a given standard Young tableau $T_{\lambda }$ only
a particular permutations $\sigma \in S(N)$ are such that $\sigma
(T_{\lambda })$ is a $\operatorname{SYT}$.  From definition~\ref{M4} it follows that $S(N)$ acts on the set of $\operatorname{YT}$ transitively and moreover we have

\begin{proposition}
\label{M41}
Choosing $\operatorname{YT} \ T_{\lambda}^1$ which is in fact $\operatorname{SYT}$ with the canonical embedding~\footnote{By $\operatorname{SYT}$ with canonical row embedding we understand Young tableau of the shape $\lambda$ filled with numbers $1,\ldots,N$ in such a way, that starting from the left-top corner we put into first box $1$, then we put $2$ into second one on the right in the same row. We continue this procedure up to $N$. Reader can see Example~\ref{example1}, where all $\operatorname{SYTs}$ for $\lambda=(2,1)$ are presented (second row). The canonical embedding is presented by first $\operatorname{SYT}$ from the left. In the similar way we can define canonical column embedding.}, we establish a bijective correspondence between $S(N)$ and the set of $\operatorname{YT}$ given by following relation
\be
\label{S1}
\forall \overline{T}_{\lambda} \ \exists ! \sigma \in S(N) \ \sigma \overline{T}^1_{\lambda}=\overline{T}_{\lambda}.
\ee
In particular we have
\be
\label{S2}
\forall T_{\lambda} \ \exists ! \sigma \in S(N) \ \sigma T^1_{\lambda}=T_{\lambda},
\ee
where symbol $\exists !$ means "there exists unique".
\end{proposition}

\begin{remark}
\label{M41A}
Obviously, in general, for a given $\overline{T}_{\lambda }$ (or $T_{\lambda
}$) the permutation $\sigma $ in the eq.~\eqref{S1},~\eqref{S2} depends on the choice of $%
T_{\lambda }^{1}$, but as we will see this dependence in not important for
the properties of the matrix $U$.
\end{remark}

In our further considerations the $\operatorname{SYT}$ will be the most important because,
as we will see, they will label the bases of the Young-Yamanouchi irreducible
representations of the symmetric group $S(N).$

\bigskip In the Young-Yamanouchi  irreducible representations of the group $%
S(N)$ the concept of axial distance play important role.

\begin{definition}
\label{M5}
The axial distance $\rho (T_{\lambda };i,j)$ between the boxes $%
i,j$ in the standard Young tableau $T_{\lambda }$ is the number of
horizontal or vertical steps to get from $i$ to $j.$ Each step is counted $%
+1 $ if it goes down or to the left and its counts $-1$ if it goes up or to
the right.
\end{definition}

\bigskip The axial distance has the following properties which follows
directly from its definition

\begin{proposition}
\label{M6}
\be
\rho (T_{\lambda };i,j)=-\rho (T_{\lambda };j,i),\qquad \rho (T_{\lambda
^{t}}^{t};i,j)=-\rho (T_{\lambda };i,j). 
\ee
\end{proposition}

The $\operatorname{SYT}$ can be characterized in simple way by so called Yamanouchi symbols
in the following way.

\begin{definition}
\label{M7}
For any standard Young tableau $T_{\lambda }$ we define a row
Yamanouchi symbol $(\operatorname{RYS}$) as a row of $n$ numbers%
\be
M_{\lambda }=(M_{1}(\lambda ),M_{2}(\lambda ),\ldots,M_{n}(\lambda )) 
\ee
where $M_{i}(\lambda )$ is the number of the row in the standard Young
tableau $T_{\lambda }$ in which the number $i$ is contained. Similarly a
column Yamanouchi symbol $(\operatorname{CYS})$ is defined also as a row of $n$ numbers 
\be
N_{\lambda }=(N_{1}(\lambda ),N_{2}(\lambda ),\ldots,N_{1}(\lambda )) 
\ee
where now $N_{i}(\lambda )$ is a number of column in $T_{\lambda }$ in wich
the number $i$ appears.
\end{definition}

From the definitions of the standard Young tableau and of the Yamanouchi
symbols it follows that for a given Young diagram $\lambda $ the row
Yamanouchi symbol $M_{\lambda }$ characterizes uniquely the corresponding
standard Young tableau $T_{\lambda }.$ Similarly we have a bijective
correspondence between the column Yamanouchi symbols $N_{\lambda }$ and the
standard Young tableau $T_{\lambda }.$ Both symbols $M_{\lambda }$ and $%
N_{\lambda }$ characterize in a unique way the corresponding Young diagram
and the standard Young tableau and in the notation of Definition~\ref{M1}. we have 
\be
T_{\lambda }=(t_{M_{t}(\lambda )N_{t}(\lambda )}^{\lambda }).
\ee

Directly from the definition of $\operatorname{RYS}$ and $\operatorname{CYS}$ we get

\begin{proposition}
\label{M8}
 Let $T_{\lambda }$ be a $\operatorname{SYT}$, $M_{\lambda }=(M_{1}(\lambda
),M_{2}(\lambda ),\ldots,M_{n}(\lambda ))$ and $N_{\lambda }=(N_{1}(\lambda
),N_{2}(\lambda ),\ldots,N_{1}(\lambda ))$ the corresponding $\operatorname{RYS}$ and $\operatorname{CYS}$
respectively. For the conjugated standard Young tableau $T_{\lambda ^{t}}^{t}
$ we denote by $M_{\lambda ^{t}}^{t}=(M_{1}(\lambda ^{t}),M_{2}(\lambda
^{t}),\ldots,M_{n}(\lambda ^{t}))$ and $N_{\lambda ^{t}}^{t}=(N_{1}(\lambda
^{t}),N_{2}(\lambda ^{t}),\ldots,N_{1}(\lambda ^{t}))$ the corresponding $\operatorname{RYS}$
and $\operatorname{CYS}$. Then we have 
\be
M_{\lambda ^{t}}^{t}=N_{\lambda },\qquad N_{\lambda ^{t}}^{t}=M_{\lambda },
\ee
e.i. the $\operatorname{RYS}$ (respectively $\operatorname{CYS}$) for the conjugated standard Young
tableau $T_{\lambda ^{t}}^{t}$ its equal to $\operatorname{CYS}$ (respectively $\operatorname{RYS}$) of $%
T_{\lambda }$ .
\end{proposition}

\bigskip The advantage of description of $\operatorname{SYT}$ in terms of Yamanouchi
symbols is that one can easily introduce the linear (lexicographic) ordering in the set of
all $\operatorname{RYS}$ (respectively $\operatorname{CYS})$ symbols for a given Young diagram $\lambda .$
In fact we have

\begin{definition}
\label{M9}
Let $M_{\lambda }=(M_{1}(\lambda ),M_{2}(\lambda
),\ldots,M_{n}(\lambda ))$ and $M_{\lambda }^{\prime }=(M_{1}^{\prime }(\lambda
),M_{2}^{\prime }(\lambda ),\ldots,M_{n}^{\prime }(\lambda ))$ be two $\operatorname{RYS}$
then $M_{\lambda }$ is smaller then $M_{\lambda }^{\prime }$ which will be
denoted $M_{\lambda }<M_{\lambda }^{\prime }$ if 
\be
\exists j\in \{1,2,\ldots,n\}:M_{i}(\lambda )=M_{i}^{\prime }(\lambda )\quad
i<j\quad \wedge \quad M_{j}(\lambda )<M_{j}^{\prime }(\lambda ), 
\ee
and similarly for $\operatorname{CYS}.$
\end{definition}

Obviously the linear order in $\operatorname{RYS}$ or in $\operatorname{CYS}$ induce the linear order of $%
\operatorname{SYT}$ but these orders are not same. In fact using this definition as well
the definitions of $\operatorname{RYS}$ and $\operatorname{CYS}$ it is not difficult to prove the following
statement

\begin{proposition}
\label{M10}
Let $\lambda =(\lambda _{1},\lambda _{2}, \ldots ,\lambda _{k})$ be
a Young diagram and suppose that all $\operatorname{RYS}$ describing all $\operatorname{SYT}$ for $\lambda 
$ are ordered in the following way%
\be
\label{Ord1}
M_{\lambda }^{1}<M_{\lambda }^{2}< \ldots <M_{\lambda }^{k} 
\ee
then 
\be
N_{\lambda }^{1}>N_{\lambda }^{2}> \ldots >N_{\lambda }^{k}, 
\ee
where $M_{\lambda }^{i}$ and $N_{\lambda }^{i}$ are respectively $\operatorname{RYS}$ and 
$\operatorname{CYS}$ of standard Young tableau $T_{\lambda }^{i}.$ Thus the ordering of $%
\operatorname{SYT}$ induced by the linear ordering of $\operatorname{RYS}$ is opposite to the ordering of
the $\operatorname{SYT}$ induced by the linear ordering of $\operatorname{CYS}$. One can see that in~\eqref{Ord1} symbol $M_{\lambda }^{1}$ corresponds with the canonical row embedding, while $M_{\lambda }^{k}$ with the canonical column embedding.
\end{proposition}

It is clear that the action \ of the group $S(N)$ on $\operatorname{SYT}$ induce the action
of $S(N)$ on $\operatorname{RYS}$ and $\operatorname{CYS}$ of $\operatorname{SYT}$ and if standard Young tableau $%
T_{\lambda }$ (with $M_{\lambda }$ and $N_{\lambda }),$ $T_{\lambda
}^{\prime }$ (with $M_{\lambda }^{\prime }$ and $N_{\lambda }^{\prime })$ $%
\sigma \in S(N)$ are such that 
\be
\sigma (T_{\lambda })=T_{\lambda }^{\prime }, 
\ee
then 
\be
\sigma (M_{\lambda })=M_{\lambda }^{\prime },\qquad \sigma (N_{\lambda
})=N_{\lambda }^{\prime }. 
\ee
From Proposition~\ref{M41} and Proposition~\ref{M10} it follows

\begin{proposition}
\label{M11}
Let $\lambda =(\lambda _{1},\lambda _{2},\ldots,\lambda _{k})$ be a
Young diagram and suppose that all $\operatorname{RYS}$ describing all $\operatorname{SYT}$ for $\lambda $
are linearly ordered 
\be
M_{\lambda }^{1}<M_{\lambda }^{2}< \ldots <M_{\lambda }^{k},
\ee
where $M^1_{\lambda}$ corresponds to $T^1_{\lambda}$.
Then for any $M_{\lambda }^{i}$, $i=1,2,\ldots,k$ there exists unique
permutation $\sigma_{i}\in S(N)$ such that 
\be
\label{M11eq2}
M_{\lambda }^{i}=\sigma_{i}(M_{\lambda }^{1}),\qquad \sigma_{1}\equiv id. 
\ee
So we have a bijective concordance between the set of all $\operatorname{SYT}$ of a given
Young diagram $\lambda $ and a subset of permutations in $S(N).$
\end{proposition}

Now we describe a construction of Young-Yamanouchi irreducible
representations of the group $S(N).$ The dimension of an irreducible
representation of $S(N)$ indexed by a partition $\lambda $ is determined by
the corresponding Young diagram and \ it will be denoted $d_{\lambda }.$ The
construction of irreducible representations of $S(N)$ is based on the fact,
that the basis vectors of the representation space may be indexed by the set
of $\operatorname{SYT}$ for the Young diagram $\lambda .$ Because, for a given $\lambda $
any standard Young tableau $T_{\lambda }$ may be described uniquely by the
corresponding row Yamanouchi symbol $M_{\lambda }$, so the basis vectors in
the representation space may be labelled by $M_{\lambda }$, which
additionally introduce a linear order in the set of basis vectors in the
representation space (see Def.~\ref{M9} and Prop.~\ref{M10}). So the orthonormal basis of
the representation space for the Young diagram $\lambda $ will be denoted in
the following way%
\be
\label{indexation}
\{e^{\lambda}_{T^i_{\lambda}} \equiv e_{M_{\lambda }^{i}}^{\lambda }:i=1,2,\ldots,d_{\lambda }\},
\ee
 and the order of the basis is induced by the order of the $\operatorname{RYS}$ $M_{\lambda
}^{1}<M_{\lambda }^{2}< \ldots <M_{\lambda }^{d_{\lambda }}.$

It is known that the symmetric group $S(N)$ is generated by the
transpositions of the form $(k \ k+1),$ $k=1,2,\ldots,n-1$, thus in order to define
a representation of $S(N)$ 
\be
D^{\lambda }:S(N)\rightarrow \operatorname{Hom}(\operatorname{span}_{%
\mathbb{C}
}\{e_{M_{\lambda }^{i}}^{\lambda }:i=1,2,\ldots ,d_{\lambda }\}) 
\ee

it is enough to define the representation operators for the generators $%
(k \ k+1)$ only. By definition these generators acts on the basis vectors $%
\{e_{M_{\lambda }^{i}}^{\lambda }:i=1,2,\ldots ,d_{\lambda }\}$ in the following
way 
\be
\label{Mx}
D^{\lambda }(k \ k+1)(e_{M_{\lambda }^{i}}^{\lambda })=\rho^{-1} (T_{\lambda
};k+1 \ k)e_{M_{\lambda }^{i}}^{\lambda }+\sqrt{1-\rho^{-2} (T_{\lambda
};k+1 \ k)}e_{(k \ k+1)M_{\lambda }^{i}}^{\lambda }
\ee
where the second term on $\operatorname{RHS}$ appears only if $(k \ k+1)T_{\lambda }$ is a $%
\operatorname{SYT}$, in this case $(k \ k+1)M_{\lambda }^{i}=M_{\lambda }^{j},$ $%
j=1,\ldots ,d_{\lambda }.$

It is known that for any irreducible representations $D^{\lambda }$ of $S(N)$
the composition of representations $\sgn D^{\lambda }$ is also a
representation of $S(N)$ and moreover we have the following 
\begin{lemma}~\cite{Procesi}
\label{SN}
Suppose that we are given with two inequivalent irreducible representations $D^{\lambda}$ and $\operatorname{sgn} D^{\lambda^t}$, where $\lambda^{t}$ denotes dual partition to $\lambda$. Then irreducible representations  $\operatorname{sgn} D^{\lambda^t}$ and $D^{\lambda}$ are isomorphic. 
\end{lemma}
From this lemma it follows that 
\be
\exists U\in U(d_{\lambda})\quad \forall \sigma \in S(N)\quad D^{\lambda
^{t}}(\sigma )=\operatorname{sgn}(\sigma )UD^{\lambda }(\sigma )U^{\dagger}.
\ee
The examples calculated in the previous section suggest that this matrix $U$
has a very simple anti-diagonal form with $\pm 1$ on the anti-diagonal as in equation~\eqref{formU}. Now
we are ready to prove the this hypothesis.

We have two irreducible representations $D^{\lambda }$ and $D^{\lambda ^{t}}$
of the group $S(N)$ acting respectively in the representation spaces $%
\operatorname{span}_{\mathbb{C}}\{e_{M_{\lambda }^{i}}^{\lambda }:i=1,2,\ldots ,d_{\lambda }\}$ and $%
\operatorname{span}_{\mathbb{C}}\{e_{M_{\lambda ^{t}}^{i}}^{\lambda ^{t}}:i=1,2,\ldots,d_{\lambda }\}$ of the
same dimension. From Prop.~\ref{M8} we get 
\be
e_{M_{\lambda ^{t}}^{i}}^{\lambda ^{t}}=e_{N_{\lambda }}^{\lambda
^{t}}:i=1,2,\ldots ,d_{\lambda }
\ee
and from Prop.~\ref{M10} \ it follows that the base $\{e_{M_{\lambda
^{t}}^{i}}^{\lambda ^{t}}:i=1,2,\ldots ,d_{\lambda }\}$ has an oposite order with
respect to the order of the basis $\{e_{M_{\lambda }^{i}}^{\lambda
}:i=1,2,\ldots,d_{\lambda }\}.$ Now let us consider a unitary transformation
between these bases 
\be
U(e_{M_{\lambda }^{i}}^{\lambda })\equiv \operatorname{sgn}(\sigma_{i})e_{N_{\lambda
}^{i}}^{\lambda ^{t}}=\operatorname{sgn}(\sigma_{i})e_{M_{\lambda ^{t}}^{i}}^{\lambda ^{t}}
\ee
where $\sigma_{i}$ are defined in Prop.~\ref{M11}. 

\begin{remark}
\label{rem_new1}
 Using the isomorphism $V^{\ast }\otimes V \simeq \operatorname{End}(V),$ where $V$ is a
linear space and $V^{\ast }$ is a dual of $V$, the unitary transformation $U$
may be written in the following operator form%
\be
U=\sum_{M_{\lambda }^{i}}\sgn(\sigma _{i})e_{M_{\lambda }^{i}}^{\ast \lambda
}\otimes e_{M_{\lambda ^{t}}^{i}}^{\lambda ^{t}}=\sum_{T_{\lambda
}^{i}}\sgn(T_{\lambda }^{i})|T_{\lambda ^{t}}^{i}\>\<T_{\lambda }^{i}|
\ee
where $\sgn(T_{\lambda }^{i})=\sgn(\sigma _{i})$ (Props.~\ref{M41},~\ref{M11}) and in the
last equation we have introduced a physical $"bra","ket"$ notation $%
e_{M_{\lambda ^{t}}^{i}}^{\lambda ^{t}}\equiv |T_{\lambda }^{i}\>.$
Note also if we chose another $\operatorname{SYT}$ as $T_{\lambda }^{1}$ in Prop.~\ref{M41}, then the corresponding matrix $U$ will differ from the
initial one by a global sign but the similarity transformation defined by
these matrices will be the same.
\end{remark}

The action of $U$ on both sides
of the equation~\eqref{Mx} gives 
\be
\begin{split}
UD^{\lambda }(k \ k+1)U^{\dagger}\operatorname{sgn}(\sigma_{i})(e_{M_{\lambda ^{t}}^{i}}^{\lambda
^{t}})=&\rho^{-1} (T_{\lambda };k+1 \ k)\operatorname{sgn}(\sigma_{i})e_{M_{\lambda
_{t}}^{i}}^{\lambda ^{t}}+ \\
&+\sqrt{1-\rho^{-2} (T_{\lambda };k+1 \ k)}\operatorname{sgn}((k \ k+1)\sigma
_{i})e_{(k \ k+1)M_{\lambda ^{t}}^{i}}^{\lambda ^{t}}.
\end{split}
\ee
Using the properties of the representation $\operatorname{sgn}$ and Prop.~\ref{M6} we get 
\be
UD^{\lambda }(k \ k+1)U^{\dagger}(e_{M_{\lambda ^{t}}^{i}}^{\lambda ^{t}})=-\rho^{-1}
(T_{\lambda ^{t}}^{t};k+1 \ k)e_{M_{\lambda _{t}}^{i}}^{\lambda ^{t}}-%
\sqrt{1-\rho^{-2} (T_{\lambda ^{t}}^{t};k+1 \ k)}e_{(k \ k+1)M_{\lambda
^{t}}^{i}}^{\lambda ^{t}}
\ee
which means that 
\be
-UD^{\lambda }(k \ k+1)U^{\dagger}=D^{\lambda ^{t}}(k \ k+1),\quad k=1,2,\ldots ,n-1
\ee
and consequently we get 
\be
\forall \sigma \in S(N)\quad D^{\lambda ^{t}}(\sigma )=\operatorname{sgn}(\sigma
)UD^{\lambda }(\sigma )U^{\dagger}.
\ee
In the bases $\{e_{M_{\lambda }^{i}}^{\lambda }:i=1,2,\ldots ,d_{\lambda }\}$ and 
$\{e_{M_{\lambda ^{t}}^{i}}^{\lambda ^{t}}:i=1,2,\ldots ,d_{\lambda }\}$ the
operator $U$ takes the  matrix form as in equation~\eqref{formU},
so one can state 
\begin{theorem}
\label{hypo1}
The operator unitary $U$ which defines a similarity transformations between
two conjugated \ Young-Yamanouchi irreducible representations of $S(n)$ with
bases $\{e_{M_{\lambda }^{i}}^{\lambda }:i=1,2,\ldots,d_{\lambda }\}$ and $%
\{e_{M_{\lambda ^{t}}^{i}}^{\lambda ^{t}}:i=1,2,\ldots,d_{\lambda }\}$  may be
written in the following way%
\be
U=\sum_{M_{\lambda }^{i}}\sgn(\sigma _{i})e_{M_{\lambda }^{i}}^{\ast \lambda
}\otimes e_{M_{\lambda ^{t}}^{i}}^{\lambda ^{t}}=\sum_{T_{\lambda
}^{i}}\sgn(T_{\lambda }^{i})|T_{\lambda ^{t}}^{i}\>\<T_{\lambda }^{i}|
\ee
where $\sgn(T_{\lambda }^{i})=\sgn(\sigma _{i})$ (see Prop.~\ref{M11}) and the relation between Yamanouchi symbols $M_{\lambda }^{i}$ and $SYT
$ $T_{\lambda }^{i}$ is described in Def.~\ref{M7} (see also Remark~\ref{rem_new1}). If  the bases $\{e_{M_{\lambda }^{i}}^{\lambda }:i=1,2,\ldots,d_{\lambda }\}$
and $\{e_{M_{\lambda ^{t}}^{i}}^{\lambda ^{t}}:i=1,2,\ldots,d_{\lambda }\}$ are
ordered according lexicographic order (see 
Def.~\ref{M9} and Prop.~\ref{M10}), then the matrix of $U$
has the following  form 
\be
U=\left( 
\begin{array}{cccccc}
0 & 0 & . & . & . & \sgn(\sigma _{1}) \\ 
0 &  &  &  & \sgn(\sigma _{2}) & 0 \\ 
. &  &  & . &  & . \\ 
. &  & . &  &  & . \\ 
0 & \sgn(\sigma _{d_{\lambda }-1}) &  &  &  & 0 \\ 
\sgn(\sigma _{d_{\lambda }}) & 0 & . & . & 0 & 0%
\end{array}%
\right) .
\ee
\end{theorem}

One can see that we can also reformulate above statement without referring to particular ordering of Young-Yamanouchi basis, namely we can write:
\begin{proposition}
\label{MH}
Similarity matrix $U$ can be also written in the following form
\be
U=\sum_{T_{\lambda}} \sgn(T_{\lambda})|T_{\lambda^t}\>\< T_{\lambda}|,
\ee
where $\sgn(T_{\lambda})=\sgn(\sigma)$ and $\sigma$ is the permutation that transforms arbitrary chosen, fixed $\operatorname{SYT}$ $T_{\lambda}^{\times}$ into $T_{\lambda}$ (see Prop.~\ref{M41} and Rem.~\ref{M41A}). For different choices of $T_{\lambda}^{\times}$, the corresponding $U$ may differ by a global sign.
\end{proposition}

\section{Applications}
\label{appl}
In this section we present some mathematical application of isomorphism between $D^{\lambda}$ and $\mathbf{\sgn} D^{\lambda}$ (see Section~\ref{symm2}). Let us consider $n-$particle system on Hilbert spaces $\mathcal{H}_A^{\ot n}$ and $\mathcal{H}_B^{\ot n}$, then full Hilbert space o such system is of course $\mathcal{H}_A^{\ot n} \ot \mathcal{H}_B^{\ot n}=\left(\mathcal{H}_A \ot \mathcal{H}_B\right)^{\ot n}\equiv \mathcal{H}_{AB}^{\ot n}$. Our goal here is to study antisymmetric projector $\operatorname{P}_{as}$ and decompose it into direct sum of some smaller projectors, which can be studied separately. Suppose that $\mathcal{H}_{A(B)}^{\ot n}=\left(\mathbb{C}^d\right)^{\ot n}$, then we can define some special class of operators called permutation operators $V$ in the following way
\begin{definition}
\label{def11}
$V:$ $S(n)$ $\rightarrow \Hom(\mathcal{(%
\mathbb{C}
}^{d})^{\otimes n})$ and 
\be
\forall \sigma \in S(n)\qquad V_{\sigma}.e_{i_{1}}\otimes e_{i_{2}}\otimes
..\otimes e_{i_{n}}=e_{i_{\sigma ^{-1}(1)}}\otimes e_{i_{\sigma
^{-1}(2)}}\otimes ..\otimes e_{i_{\sigma ^{-1}(n)}}, 
\ee
where $d\in 
\mathbb{N}
$ and $\{e_{i}\}_{i=1}^{d}$ is an orthonormal basis of the space $\mathcal{%
\mathbb{C}
}^{d}.$
\end{definition}
Now using above definition we are ready to formulate
\begin{fact}
Antisymmetric projector $\operatorname{P}_{as}$ acting on $\mathcal{H}_{AB}^{\ot n}$ can be written in the form
\be
\operatorname{P}_{as}=\bigoplus_{\lambda \vdash n}\mathbf{1}_{\lambda}^A \ot \mathbf{1}_{\lambda^t}^B\ot W_{\lambda \lambda^t}^{AB},
\ee
where identities and operator $W_{\lambda \lambda^t}^{AB}$ act on unitary and symmetric part respectively. By $\lambda^{t}$ we denote conjugate Young diagram and operator $W_{\lambda \lambda^t}^{AB}=|\psi\>\<\psi|_{AB}^{\lambda, \lambda^t}$, where
\be
\label{W}
|\psi\>_{AB}^{\lambda,\lambda^t}=\sum_{T_{\lambda}} \sgn(T_{\lambda})|T_{\lambda}\>_A |T_{\lambda^t}\>_{B}.
\ee
Interpretation of  $\sgn(T_{\lambda})$ is given in Proposition~\ref{MH}. 
\end{fact}

\begin{proof}
Proof is based on direct calculations and Schur-Weyl duality described for example in~\cite{RWallach}. First of all let us write projector on antisymmetric subspace of $\mathcal{H}_{AB}^{\ot n}$ as
\be
\label{beg}
\operatorname{P}_{as}=\frac{1}{n!}\sum_{\sigma \in S(n)} \operatorname{sgn}(\sigma)V_{\sigma}^{A}\ot V_{\sigma}^B.
\ee
Thanks to Schur-Weyl duality we can decompose every permutation operator into direct sum of irreducible components
\begin{equation}
\label{decompSwap}
V_{\sigma}^A=\bigoplus_{\lambda_A}\mathbf{1}_{\lambda_A}\ot \left(V_{\sigma}^A \right)_{\lambda_A},\quad V_{\sigma}^B=\bigoplus_{\lambda_B}\mathbf{1}_{\lambda_B}\ot \left(V_{\sigma}^B \right)_{\lambda_B}.
\end{equation}
Now putting equations~\eqref{decompSwap} into equation~\eqref{beg} we obtain
\be
\operatorname{P}_{as}=\bigoplus_{\lambda_A,\lambda_B}\mathbf{1}_{\lambda_A}\ot \mathbf{1}_{\lambda_B}\ot \left[\frac{1}{n!}\sum_{\sigma\in S(n)}\operatorname{sgn}(\sigma)\left(V_{\sigma}^A \right)_{\lambda_A}\ot \left(V_{\sigma}^B \right)_{\lambda_B} \right]=\bigoplus_{\lambda_A,\lambda_B}\mathbf{1}_{\lambda_A}\ot \mathbf{1}_{\lambda_B}\ot \operatorname{P}^{\lambda_A \lambda_B}_{as}.
\ee
Let us define the following.  For an arbitrary matrix $C$, the state $\Psi[C]$, which is defined by taking matrix elements $c_{ij}$ of $C$ an setting them as coefficients in standard basis $|i\> \ot |j\> $  on subsystems $A$ and $B$, i.e. $\Psi[C]=\sum_{ij}c_{ij}|i\> \ot |j\>$. Now
thanks to property
\begin{equation}
\label{prop}
X\ot Y \Psi[C]=\Psi\left[XCY^T \right],\quad X,Y-\text{arbitrary matrices},
\end{equation}
we can write 
\be
\operatorname{P}_{as}^{\lambda_A \lambda_B}\Psi[C]=\frac{1}{n!}\Psi\left[\sum_{\sigma\in S(n)}\operatorname{sgn}(\sigma) \left(V_{\sigma}^A \right)_{\lambda_A}C\left(V_{\sigma}^{B} \right)_{\lambda_B}^T\right].
\ee
Using considerations from previous sections we see, that for $\lambda_B=\lambda_A^t$ there is always nonsingular matrix $U$ such that
$U\left(V_{\sigma}^{B} \right)_{\lambda_B=\lambda_A^{t}}U^{\dagger}=\operatorname{\text{sgn}}(\sigma) \left(V_{\sigma}^A \right)_{\lambda_A}$. Because of this equivalence and Schur lemma we can write, that $C \sim U$. Now taking into account formula~\eqref{prop} we have $\Psi[U]=\sum_{ij}u_{ij}|i\>\ot |j\>$, where $u_{ij}$ are matrix elements of unitary transformation $U$ (see Proposition~\ref{MH}), so
\begin{equation}
\operatorname{P}_{as}=\frac{1}{n!}\bigoplus_{\lambda}\mathbf{1}_{\lambda}^A\ot \mathbf{1}_{\lambda^t}^B\ot W_{\lambda \lambda^t}^{AB},
\end{equation}
where operator $W_{\lambda \lambda^t}^{AB}$ is given by equation~\eqref{W}.
\end{proof}
Finally is worth to mention that similar decomposition was done for symmetric projector in~\cite{Matsumoto1}, where authors consider entanglement concentration for many copies of unknown pure states and propose some protocol which produces perfect maximally entangled state.

\section{Conclusions}
In this paper we present and discuss explicit method of constructing
unitary maps between two arbitrary but equivalent irreducible
representations of some finite group $G$ (Lemma~\ref{problem:1}).  We
observe a few interesting properties in the general case, such as the
doubly stochastic property (Remark~\ref{stoch}) or and a
generalization of the classical orthogonality relation for irreducible
representations (Corollary~\ref{ortRel}). In the next part  we apply
of our method to the symmetric group $S(N)$
(Example~\ref{exx1},~\ref{exx2},~\ref{exx3},~\ref{exx4},~\ref{exx5}
and finally~\ref{exx6}) which give us the clue that whenever we use as
a basis the Young-Yamanouchi basis our transformation matrices $U$
between Young-Yamanouchi conjugated irreps have anti-diagonal form
with $\pm 1$ (see Theorem~\ref{hypo1}).    We hope that our results
will be useful for numerical work involving $S(N)$ and other group
symmetry.

\section{Acknowledgment}
The authors would like to thanks Aram W. Harrow, Mary Beth Ruskai for numerous discussions and to Josh Grochow for useful references.
The authors would like to thank also to Issac Newton Institute for
hospitality where some part of this work was done. M.~S.~is supported
by the International PhD Project "Physics of future quantum-based
information technologies": grant MPD/2009-3/4 from Foundation for
Polish Science and Grant NCN Maestro (DEC- 2011/02/A/ST2/00305). M.~M. and M.~H. are supported by MNiSW Ideas-Plus
Grant(IdP2011000361).  
Part of this work was done in National Quantum
Information Centre of Gda\'nsk.

\end{document}